\documentclass{amsart}

\usepackage{amsmath}
\usepackage{amssymb}
\usepackage{xcolor}

\def\pa{\partial}
\def\var{\varepsilon}
\def\ov{\overline}
\def\R{{\Bbb R}}

\newtheorem{Thm}{Theorem}
\newtheorem{Prop}[Thm]{Proposition}

\def\N{{\Bbb N}}

\begin{document}

\title[4-species diffusive model]{Simple proofs for the existence of smooth
  solutions to a reaction-diffusion system modeling reversible chemistry}

\author[H. Bouton]{Hector Bouton$^{(1)}$}
\author[L. Desvillettes]{Laurent Desvillettes$^{(2)}$}
\author[H. Dietert]{Helge Dietert$^{(3)}$}

\address{$(1)$ Universit\'e Paris Cit\'e and Sorbonne Universit\'e, CNRS, Institut de Math\'ematiques de Jussieu-Paris Rive Gauche}
\email{bouton@imj-prg.fr}
\address{$(2)$ Universit\'e Paris Cit\'e and Sorbonne Universit\'e, CNRS, IUF, Institut de Math\'ematiques de Jussieu-Paris Rive Gauche}
\email{desvillettes@math.univ-paris-diderot.fr}
\address{$(3)$ Universit\'e Paris Cit\'e and Sorbonne Universit\'e, CNRS, Institut de Math\'ematiques de Jussieu-Paris Rive Gauche}
\email{helge.dietert@imj-prg.fr}


\begin{abstract}
  We present in this work a very short proof for the existence, uniqueness and
  smoothness in dimensions $d\le3$ of the system of reaction diffusion
  $ \pa_t a_i - d_i \, \Delta a_i = (-1)^i \,(a_1\,a_3 - a_2\,a_4)$, where $a_i \ge 0$
  model the concentrations of chemical species undergoing a chemical reaction
  and diffusing (each with its diffusion rate $d_i>0$) in a bounded container.
\end{abstract}

\maketitle

\noindent{\small{\textbf{Classification AMS 2010:} 35A01, 35A09, 35K57, 35Q92.}}
	\medskip


\section{Introduction} \label{intro}

 In this paper, we consider the system of reaction diffusion 
\begin{equation}\label{pp1}
\pa_t a_i - d_i \Delta a_i = (-1)^i\, (a_1\,a_3 - a_2\, a_4),
\end{equation}
\begin{equation}\label{pp2}
 \forall t \ge0, x\in \pa\Omega, \qquad  \nabla a_i \cdot n = 0 , 
\end{equation}
with initial data $a_i^{in} \ge 0$. Here $a_i := a_i(t,x) \ge 0$, $i=1,\dots,4$, are the
unknowns of the system, and the variables are $t \ge 0$ and $x\in \Omega$, where $\Omega$ is a
bounded smooth open set of $\R^d$, $d \in \N - \{0\}$. The diffusion rates $d_i$
are strictly positive, and $n = n(x)$ is the outward unit normal vector at point
$x \in \pa\Omega$.  \medskip

This system models, according to the mass action law, the evolution of 4
chemical species $A_1,\dots,A_4$, which diffuse in a bounded container
(represented by $\Omega$) in which they are confined (whence the homogeneous Neumann
boundary condition), and undergo the reversible reaction $A_1 + A_3 \rightleftharpoons A_2 + A_4$
(for presentation, we assume that the reaction rates are $k=l=1$). We focus on
dimensions $d\le3$, which are the most interesting for applications. 
 \medskip

This system has attracted a lot of attention in the past years. Weak solutions
in all dimension were known to exist at least since the 2000s (cf.\ \cite{DFPV},
and in a more general context \cite{fischer}), while strong (unique) solutions
were built in \cite{CDF} in dimension 1 and 2. In higher dimension, such strong
solutions were provided in the papers \cite{souplet}, \cite{CGV} and \cite{FMT}
(cf.\ also the pioneering work of \cite{kanel}).  \medskip

In this work, we intend to provide a complete self-contained proof of existence
to strong solutions to system \eqref{pp1}, \eqref{pp2} in dimension \(d\le 3\),
which is very short. It uses heavily the entropy structure of the equation,
together with some ideas coming from our previous work \cite{BDDmain}, including
an interpolation inequality that we present here in detail using only elementary
computations.
Note that this entropy structure has already been exploited in \cite{souplet} and in \cite{Tan18} and \cite{STY23} in low dimension. We think that the proof presented here is quite robust and we
intend to illustrate this robustness in a forthcoming work devoted to a
degenerate case \cite{BDDdeg}. 
 \medskip

To be precise, we present here a short proof of the following result (for which
already existing proofs can be found in \cite{souplet}, \cite{CGV} and
\cite{FMT}):

\begin{Thm}\label{ercwd}
  Let $\Omega \subset \R^d$, \(d \le 3\), be a smooth bounded open set (with $n(x)$ the
  outward unit normal vector at point $x \in \pa\Omega$), and
  $d_i>0, a_i^{in} = a_i^{in}(x) \ge 0$ ($i=1,\dots,4$), be diffusion rates and
  smooth $C^{\infty}({\ov{\Omega}})$ initial data which are compatible with the
  homogeneous Neumann boundary conditions.

  Then there exists a unique smooth
  $(a_i)_{i=1,\dots,4} \in C^{\infty}(\R_+ \times {\ov{\Omega}})$ solution to the parabolic
  problem
  \begin{equation}\label{pp1n}
    \pa_t a_i - d_i \Delta a_i = (-1)^i\, (a_1\,a_3 - a_2\, a_4),
  \end{equation}
  \begin{equation}\label{pp2n}
    \forall t \ge 0,  x\in \pa\Omega, \qquad  \nabla a_i \cdot n = 0 , 
  \end{equation}
  for \(i=1,\dots,4\), and initial data $a_i^{in}$.
\end{Thm}

The paper is structured as follows: First we state an interpolation inequality
in Section~\ref{sec2}, and prove it by elementary means. Then Section \ref{sec3}
is devoted to the main a priori estimate for the system. Finally, we build
strong solutions to it in Section~\ref{sec4}, and conclude there the proof of
Theorem~\ref{ercwd}. Note that the two first sections present results which hold
in any dimension of space.

\section{Interpolation inequalities} \label{sec2}

This section is devoted to the proof of the interpolation inequality below:

\begin{Prop}\label{incwd}
  Let $\Omega \subset \R^d$ be a %
Lipschitz bounded open set (with $|\Omega|$ its measure, and $n(x)$
  the outward unit normal vector at point $x \in \pa\Omega$), and $C>0$.  We suppose
  that $E = E(x)$ and $w = w(x)$ are smooth functions defined on $\bar \Omega$
  satisfying
  \begin{equation}\label{hypo1}
    \forall x\in \Omega, \qquad 0 \le E(x) \le \Delta w(x) + C, 
  \end{equation}
  \begin{equation}\label{hypo2}
    \nabla w(x) \cdot n = 0, \qquad \text{a.e.~ in } \partial\Omega. 
  \end{equation}
  Then the following interpolation inequality holds:
  \begin{equation}\label{estip1}
    \int_{\Omega} E^3 \le  20\, \|w\|_{\infty} \,   \int_{\Omega} |\nabla E|^2 + 9 C^3\, |\Omega| . 
  \end{equation}
\end{Prop}

\begin{proof}[Proof of Proposition~\ref{incwd}]
  Let \(\bar w = w - \inf_{\Omega} w\), we see that \(\bar w\) also satisfies the
  assumptions of the proposition and additionally guarantees \(\bar w \ge 0\).

  Using the bound \eqref{hypo1}, we first observe that
  \begin{equation}\label{eq:proof-interpolation:e3}
    \int_{\Omega} E^3
    \le \int_{\Omega} E^2 \Delta \bar w + C\, \int_{\Omega} E^2.
  \end{equation}

  The first term can be estimated by partial integration (recalling
  \eqref{hypo2}) as
  \begin{equation}\label{eq:proof-interpolation:e2lw}
    \int_{\Omega} E^2 \Delta \bar w
    = - 2 \int_{\Omega} E \nabla E \cdot \nabla \bar w
    \le \frac{1}{\| \bar w \|_{\infty}}
    \int_{\Omega} E^2 |\nabla \bar w|^2
    + \| \bar w \|_{\infty} \int_{\Omega} |\nabla E|^2.
  \end{equation}

  We call the first integral \(B\) and estimate it as
  \begin{align*}
    B
    &:= \int_{\Omega} E^2 |\nabla \bar w|^2 
    = - \int_{\Omega} \bar w \nabla \cdot (E^2 \nabla \bar w)\\
    &= - \int_{\Omega} \bar w E^2 \Delta \bar w
      - 2 \int_{\Omega} \bar w E\, \nabla E \cdot \nabla \bar w\\
    &\le - \int_{\Omega} \bar w E^2 (\Delta \bar w +C)
      + \| \bar w \|_{\infty} C \int_{\Omega} E^2
      + \frac{1}{2} \int_{\Omega} E^2 |\nabla \bar w|^2
      + 2 \| \bar w \|_{\infty}^2 \int_{\Omega} |\nabla E|^2.
  \end{align*}
  As \(\bar w \ge 0\) and \(\Delta \bar w + C \ge 0\), the first term has a good sign.
  The third term is \(B/2\) and can thus be absorbed.  Hence we arrive at
  \begin{equation*}
    B \le 2\, \lVert \bar w\rVert_{\infty}
    \left[ C \int_{\Omega} E^2 
    + 2\, \|\bar w\|_{\infty} \int_{\Omega} |\nabla E|^2 \right].
  \end{equation*}

  Plugging the result for \(B\) into \eqref{eq:proof-interpolation:e2lw} for
  \(\int E^2 \Delta w\) yields
  \begin{equation*}
    \int_{\Omega} E^2\Delta \bar w
    \le 
    2C \int_{\Omega} E^2
    + 5\, \| \bar w\|_{\infty} \int_{\Omega} |\nabla E|^2.
  \end{equation*}

  Hence for \(\int E^3\), we find by \eqref{eq:proof-interpolation:e3} that
  \begin{equation*}
    \int_{\Omega} E^3 \le 3C \int_{\Omega} E^2
    + 5\, \| \bar w \|_{\infty} \int_{\Omega} |\nabla E|^2.
  \end{equation*}

  By Young's inequality, we find that \(3EC \le \frac{1}{2} E^2 + \frac{9}{2} C^2\) so that
  \begin{equation*}
    \int_{\Omega} E^3 \le 9 C^2 \int_{\Omega} E
    + 10\, \| \bar w \|_{\infty} \int_{\Omega} |\nabla E|^2.
  \end{equation*}
  Then, noticing that \(\int_{\Omega} E \le C |\Omega|\) thanks to \eqref{hypo1} and observing that
 \(\| \bar w \|_{\infty} \le 2 \| w \|_{\infty}\), we arrive at
  the claimed result.
\end{proof}

\section{A priori estimates} \label{sec3}

In this section, we prove our main a priori estimate, presented in the
proposition below:

\begin{Prop}\label{eapcwd}
  Let $\Omega \subset \R^d$ be a smooth bounded open set (with $n(x)$ the outward unit
  normal vector at point $x \in \pa\Omega$), and $d_i>0, a_i^{in} = a_i^{in}(x) \ge 0$
  ($i=1,\dots,4$), be diffusion rates and smooth initial data which are defined on
  $\bar\Omega$ and compatible with the homogeneous Neumann boundary conditions.  We
  suppose that $T>0$ and $a_i = a_i(t,x) \ge 0$, defined on $[0,T] \times \bar \Omega$, is a
  smooth solution to the parabolic problem, for $i=1,\dots,4$,
  \begin{equation}\label{pp1}
    \pa_t a_i - d_i \Delta a_i = (-1)^i\, (a_1\,a_3 - a_2\, a_4),
  \end{equation}
  \begin{equation}\label{pp2}
    \forall t\in [0,T], x\in \pa\Omega, \qquad  \nabla a_i \cdot n = 0 , 
  \end{equation}
  with initial data $a_i^{in}$.
  \par
  Then  the following estimate holds:
  \begin{multline}\label{estap1}
    \sup_{t\in [0,T]} \int_{\Omega}  \sum_{i=1}^4  |a_i|^2 \, \lvert \ln a_i\rvert^2
    +  \int_0^T \int_{\Omega}  \sum_{i=1}^4  |a_i|^3 \, |\ln a_i|^3\, 1_{\{a_i \ge e^2\}} \\ 
    + \int_0^T  \int_{\Omega} \sum_{i=1}^4   |\nabla (a_i \ln a_i - a_i + 1)|^2
    \le C_{m},
  \end{multline}
  where $C_{m} >0$ depends only on $d$, $|\Omega|$, $T$, $\min d_i$, $\max d_i$,
  $\|a^{in}_i\|_{\infty}$ ($i=1,\dots,4$).
\end{Prop}
\medskip

 {\bf{Proof of Proposition \ref{eapcwd}}}: 
We first observe that we have the (local in $x$) entropy relation
\begin{equation}\label{entropy}
 \pa_t \left[ \sum_{i=1}^4 (a_i\,\ln a_i - a_i + 1)\, \right]  - \Delta \left[ \sum_{i=1}^4 d_i\, (a_i\,\ln a_i - a_i + 1)\, \right] = -p,
\end{equation}
where 
\begin{equation}\label{entropy_p}
 p :=  (a_1\,a_3 - a_2\, a_4)\, (\ln(a_1\,a_3) - \ln(a_2\,a_4)) +  \sum_{i=1}^4  d_i\, \frac{|\nabla a_i|^2}{a_i}  \ge 0. 
\end{equation}
 We define
\begin{equation}\label{defw}
 w := \int_0^t  \left[ \sum_{i=1}^4 d_i\, (a_i\,\ln a_i - a_i + 1)  \, \right] \ge 0,
\end{equation}
and observe that 
 \begin{equation}\label{dtw}
 \pa_t  w :=  \sum_{i=1}^4 d_i\, (a_i\,\ln a_i - a_i + 1)  \ge 0 ,
\end{equation}
and, using the entropy relation (\ref{entropy}),
 $$ \Delta  w := \int_0^t  \Delta\, \left[ \sum_{i=1}^4 d_i\, (a_i\,\ln a_i - a_i + 1)  \, \right]   $$
 \begin{equation}\label{deltaw} 
    =  E - E^{in}  + \int_0^t p,
\end{equation}
where 
\begin{equation}\label{def_entropy}
  E := \sum_{i=1}^4 (a_i\,\ln a_i - a_i + 1), \qquad  E^{in} := \sum_{i=1}^4 (a_i^{in}\,\ln a_i^{in} - a_i^{in} + 1),
\end{equation}
are respectively  the entropy density and the initial entropy density.
\medskip

As a consequence, $w$ satisfies the parabolic problem (on $[0,T] \times \Omega$)
\begin{equation}\label{npare}
 a \,\pa_t w  - \Delta w = E^{in} -  \int_0^t p, 
\end{equation}
\begin{equation}\label{npare2}
\forall  t\in [0,T],  x\in \pa\Omega, \quad  \nabla w \cdot n = 0 , \qquad \qquad w(0, \cdot) = 0 ,
\end{equation}
where 
$$ a : = \frac{\sum_{i=1}^4  (a_i\,\ln a_i - a_i + 1) }{ \sum_{i=1}^4 d_i\, (a_i\,\ln a_i - a_i + 1)}\ge (\max d_i)^{-1}.$$
Defining $ \bar w (t,x) :=  (\max d_i) \,\|E^{in}\|_{\infty} \, t $, we see that $\bar w$ satisfies  the parabolic problem
\begin{equation}\label{npareb}
  (\max d_i)^{-1}\,\pa_t  \bar w  - \Delta\bar  w = \|E^{in}\|_{\infty}, 
\end{equation}
\begin{equation}\label{npare2b}
 \forall  t\in [0,T],  x\in \pa\Omega, \quad  \nabla \bar w \cdot n = 0 , \qquad  \qquad \bar w(0, \cdot) = 0 . 
\end{equation}
Keeping in mind that $\pa_t w \ge 0$ thanks to \eqref{dtw},  we get the estimate
\begin{align*}
  (\max d_i)^{-1}\, \pa_t (\bar w - w) - \Delta (\bar w - w)
  &= \|E^{in}\|_{\infty} - [ (\max d_i)^{-1} - a] \, \pa_t  w
    - [a\, \pa_t w - \Delta w] \\
  &\ge \int_0^t p + (\|E^{in}\|_{\infty} - E^{in}) \ge 0, 
\end{align*}
so that using the boundary
and initial conditions \eqref{npare2} and \eqref{npare2b},
we see thanks to the comparison principle that for $t \in [0,T]$, $x \in \Omega$,
\begin{equation}\label{winf}
 0 \le w(t,x) \le \bar w(t, x) =  (\max d_i) \,\|E^{in}\|_{\infty} \, t \le  (\max d_i) \,\|E^{in}\|_{\infty} \, T. 
\end{equation}
Observing that thanks to \eqref{deltaw}, we have 
\begin{equation}\label{eew}
 0 \le E(t,x) \le \Delta w(t, x) +  \|E^{in}\|_{\infty} ,
\end{equation}
and keeping in mind the boundary condition in \eqref{npare2} and estimate
\eqref{winf}, we can use Proposition~\ref{incwd} (for any $t \in [0,T]$) and get the estimate
\begin{equation}\label{neee}
\int_{\Omega} E^3  \le 
K_1 \,   \int_{\Omega} |\nabla E|^2 + K_2 , 
\end{equation}
where
\begin{equation}\label{k1k2}
K_1 := 20\, (\max d_i) \,\|E^{in}\|_{\infty} \, T, \qquad K_2 := 9 \, \|E^{in}\|_{\infty}^3\, |\Omega| .
\end{equation}
\medskip

We now use an energy type estimate for system \eqref{pp1}, \eqref{pp2}.
\par
For this we introduce (when $x>0$) 
\begin{equation}\label{eff}
 f(x) := \frac12\, x^2\,( \ln x)^2 - \frac32\, x^2\, \ln x + \frac74\, x^2,
\end{equation}
so that 
$$ f'(x) = x\,\,( \ln x)^2 - 2\, x\, \ln x + 2x, \qquad f''(x) = (\ln x)^2 . $$
Note that $f$ and $f'$ can be extended by $0$ at point $x=0$ by continuity and are strictly increasing on $[0, \infty[$, and strictly positive on $]0, \infty[$. 
\par
We see that
$$ \frac{d}{dt} \sum_{i=1}^4  \int_{\Omega} f(a_i) = \sum_{i=1}^4  \int_{\Omega} f'(a_i) \, [d_i \Delta a_i + (-1)^i \, (a_1\,a_3 - a_2\,a_4)] $$
$$ =  -\sum_{i=1}^4  d_i \int_{\Omega} (\ln a_i)^2\, |\nabla a_i|^2 + \sum_{i=1}^4  \int_{\Omega} f'(a_i) \, (-1)^i \, (a_1\,a_3 - a_2\,a_4) $$
\begin{equation}\label{vtcinq}
=  -\sum_{i=1}^4  d_i \int_{\Omega} |\nabla (a_i \ln a_i - a_i + 1)|^2 + \sum_{i=1}^4  \int_{\Omega} f'(a_i) \, (-1)^i \, (a_1\,a_3 - a_2\,a_4) ,
\end{equation}
so that after integration in time
$$ \sum_{i=1}^4  \int_{\Omega} f(a_i) (T)  + \int_0^T \sum_{i=1}^4  d_i \int_{\Omega} |\nabla (a_i \ln a_i - a_i + 1)|^2 $$
\begin{equation}\label{vtsix}
= \sum_{i=1}^4  \int_{\Omega} f(a_i) (0)  + \sum_{i=1}^4  (-1)^i \int_0^T\int_{\Omega} f'(a_i) \,  (a_1\,a_3 - a_2\,a_4).
\end{equation}
At this level we estimate (recalling that $f'$ is increasing and nonnegative)
$$ \sum_{i=1}^4  (-1)^i \int_0^T\int_{\Omega} f'(a_i) \,  (a_1\,a_3 - a_2\,a_4) \le \frac12 \sum_{i=1}^4  \int_0^T\int_{\Omega} f'(a_i)\, \sum_{j=1}^4 a_j^2 $$
$$ \le  \frac12 \sum_{i=1}^4  \sum_{j=1}^4  \int_0^T\int_{\Omega} [f'(a_i) \, a_j^2 \, 1_{\{a_i \ge a_j\}} + f'(a_i) \, a_j^2 \, 1_{\{a_i < a_j\}}] $$
\begin{equation}\label{vtsept}
\le 4  \sum_{i=1}^4   \int_0^T\int_{\Omega} f'(a_i) \, a_i^2 .
\end{equation}
Finally, thanks to (\ref{vtsix}), (\ref{vtsept}), 
\begin{multline}\label{vthuit} 
  \frac14\, (\min d_i) \,  \int_0^T\int_{\Omega}  |\nabla E|^2  \le (\min d_i) \, \sum_{i=1}^4   \int_0^T\int_{\Omega}  |\nabla (a_i \ln a_i - a_i + 1)|^2  \\
  \le \sum_{i=1}^4  \int_{\Omega} f(a_i^{in}) + 4 \sum_{i=1}^4   \int_0^T
  \int_{\Omega} f'(a_i) \, a_i^2  .
\end{multline}
and estimate (\ref{neee}) ensures that
$$  \int_0^T\int_{\Omega}  E^3  \le 
K_3 + K_4 \sum_{i=1}^4  \int_0^T \int_{\Omega} f'(a_i) \, a_i^2 , $$
with
\begin{equation}\label{lllb}
 K_4 := 16 K_1\, (\min d_i)^{-1} \qquad  K_3 := 4 K_1\, (\min d_i)^{-1}\, \sum_{i=1}^4  \int_{\Omega} f(a_i^{in}) + K_2\, T. 
\end{equation}
 Observing that when $x \ge e^2$,
$$ x\, \ln x - x + 1 \ge \frac12\, x\, \ln x , \qquad f'(x)\, x^2 \le x^3\, (\ln x)^2, $$
we see for any $M \ge e^2$ that
$$  \int_0^T\int_{\Omega}  \sum_{i=1}^4  a_i^3 \, \lvert \ln a_i \rvert^3 \, 1_{\{a_i \ge e^2\}}  \le 8 \int_0^T\int_{\Omega}  E^3  $$
$$ \le 8 K_3 + 8 K_4  \sum_{i=1}^4  \int_0^T \int_{\Omega} f'(a_i) \, a_i^2 \, 1_{\{a_i \le e^2\}} 
   + 8 K_4  \sum_{i=1}^4  \int_0^T \int_{\Omega} f'(a_i) \, a_i^2 \, 1_{\{a_i \ge e^2\}} $$ 
$$ \le 8 K_3 + 32 K_4 T\, |\Omega| \, f'(e^2)\, e^4  + 8 K_4  \sum_{i=1}^4  \int_0^T \int_{\Omega}  a_i^3 \, (\ln a_i)^2 \,1_{\{a_i \ge M \ge e^2\}}  $$
$$  +\, 8 K_4  \sum_{i=1}^4  \int_0^T \int_{\Omega}  a_i^3 \, (\ln a_i)^2  \, 1_{\{M \ge a_i \ge e^2\}} $$ 
$$ \le 8 K_3 + 64 K_4 T\, |\Omega| \,  e^6
+ \frac{8 K_4}{\ln M}  \sum_{i=1}^4  \int_0^T \int_{\Omega}  a_i^3 \, \lvert \ln a_i \rvert^3 \, 1_{\{a_i \ge e^2\}}   
  +\, 8 K_4  M^3\, (\ln M)^2 \, |\Omega|\, T . $$
Taking 
$M := \max(\exp (16\, K_4), e^2)$
we end up with
$$   \int_0^T\int_{\Omega}  \sum_{i=1}^4  a_i^3 \, |\ln a_i|^3 \, 1_{\{a_i \ge e^2\}}  \le C_{n}, $$
where 
\begin{equation}\label{llla}
  C_{n} :=  16 K_3 + 128\,e^6\, K_4 \, |\Omega|\,T
  + 8\, K_4\, \max(16  K_4, 2)^2 \, |\Omega|\, T \, \max( \exp (48\, K_4), e^6) . 
\end{equation}
Using now estimates~\eqref{vtsix} and \eqref{vtsept}, together with the elementary inequality $f(x) \ge \frac18 \, x^2\, (\ln x)^2$, we see that 
$$ \frac18\,  \sup_{t\in [0,T]} \int_{\Omega} \sum_{i=1}^4 |a_i|^2 \, |\ln a_i|^2 + (\min d_j) \, \sum_{i=1}^4 \int_0^T\int_{\Omega} |\nabla (a_i \ln a_i - a_i + 1)|^2  $$
$$ \le  \sup_{t\in [0,T]}   \sum_{i=1}^4  \int_{\Omega} f(a_i)  +   \sum_{i=1}^4 d_i\, \int_0^T   \int_{\Omega} |\nabla (a_i \ln a_i - a_i + 1)|^2  $$
$$ \le  \sum_{i=1}^4  \int_{\Omega} f(a_i^{in})  + 4  \sum_{i=1}^4   \int_0^T\int_{\Omega} f'(a_i) \, a_i^2\, 1_{\{a_i \le e^2\}} + 4  \sum_{i=1}^4   \int_0^T\int_{\Omega} f'(a_i) \, a_i^2\, 1_{\{a_i \ge e^2\}} $$
$$ \le  \sum_{i=1}^4  \int_{\Omega} f(a_i^{in})  + 16\, f'(e^2)\, e^4 + 4  \sum_{i=1}^4  \int_0^T\int_{\Omega}    |a_i|^3 \, |\ln a_i|^2  \, 1_{\{a_i \ge e^2\}}   \le C_m, $$
where
\begin{equation}\label{lll}
 C_m : = \sum_{i=1}^4  \int_{\Omega} f(a_i^{in})  + 32\, e^6 + 16\, C_n . 
\end{equation}
One can check from the proof that the constant $C_m$ indeed depends on the
announced parameters of the model in the statement of Proposition~\ref{eapcwd},
thanks to formulas \eqref{lll}, \eqref{llla}, \eqref{lllb} and \eqref{k1k2}.
\hfill $\square$

\section{Existence and regularity of solutions} \label{sec4}

In this section, we present the end of the
\medskip

 {\bf{Proof of Theorem \ref{ercwd}}}:  We consider the unique smooth global (in $C^{\infty}(\R_+ \times {\ov{\Omega}})$) solution of the approximated problem
\begin{equation}\label{pp1napprox}
\pa_t a_i^\var - d_i \Delta a_i^\var = (-1)^i\,  \frac{a_1^\var\,a_3^\var - a_2^\var\, a_4^\var}{1 + \var\, \sum_{i=1}^4 (a_i^\var)^2},
\end{equation}
\begin{equation}\label{pp2napprox}
 \forall t \ge 0, x\in \pa\Omega, \qquad  \nabla a_i^\var \cdot n = 0 , 
\end{equation}
with initial data $a_i^{in}$. Note that the existence of such solutions can be obtained from example from \cite{desv}, since the right-hand side, when $\var>0$ is given, is bounded.
\medskip

One can easily check that the estimate of Proposition \ref{eapcwd} still holds
for $a_i^\var$, uniformly with respect to $\var$, so that 
\begin{equation*}
  \int_0^T \int_{\Omega}  \sum_{i=1}^4  |a_i^\var|^3 \, |\ln a_i^\var|^3\, 1_{\{a_i^\var \ge e^2\}} \le C_{m} . 
\end{equation*}
Then, we observe that if (for some $p \ge 2$, $T>0$) $ \sum_{i=1}^4 a_i^\var$ is
bounded in $L^p([0,T] \times \Omega)$ (here and later, uniformly in $\var > 0$), then
$\frac{a_1^\var\,a_3^\var - a_2^\var\, a_4^\var}{1 + \var\, \sum_{i=1}^4 (a_i^\var)^2}$
is bounded in $L^{\frac{p}{2}}([0,T] \times \Omega)$, and, thanks to the properties of the
heat kernel, in dimension $d\leq 3$, $ \sum_{i=1}^4 a_i^\var$ is bounded in $L^q([0,T] \times \Omega)$ for any
$q \ge 1$ such that $\frac1q > \frac2p - \frac25$.

Starting from $p_0 := 3$ and defining $\frac1{p_{n+1}} = \frac2{p_n} - \frac25$,
we see that $ \sum_{i=1}^4 a_i^\var$ is bounded in $L^{q}([0,T] \times \Omega)$ for $q<p_n$
(for all $n$ such that $p_n^{-1} >0$), that is
$p_1 = \frac{15}4$, $p_2 = \frac{15}2$, and $ p_3^{-1}<0$. A last use of the
properties of the heat kernel shows that for all $i=1,..,4$, $a_i^\var$ is
bounded in $C^{0,\alpha}(\R_+ \times {\ov{\Omega}})$ for some $\alpha \in ]0,1[$.

Schauder estimates used inductively finally yield that $a_i^\var$
is bounded in $C^{p}(\R_+ \times {\ov{\Omega}})$ for all $p\in \N$.  As a consequence,
$a_i^\var$ converges (up to extraction of a subsequence, and uniformly on $[0,T] \times \overline{\Omega}$ for all $T>0$ together with its derivatives of any order) towards some
$a_i \in C^{\infty}(\R_+ \times {\ov{\Omega}})$ which solves (in the classical sense) the original
system.
\medskip

Uniqueness can be obtained by a direct computation of stability: if $a_i^{(1)}$
and $a_i^{(2)}$ are two classical solutions of the original system with initial
data $a_i^{(1), in}$ and $a_i^{(2), in}$, then
\begin{multline*}
  \pa_t (a_i^{(1)} - a_i^{(2)}) - d_i\, \Delta (a_i^{(1)} - a_i^{(2)}) \\
  = (-1)^i \, [(a_1^{(1)} - a_1^{(2)})\, a_3^{(1)} + (a_3^{(1)} - a_3^{(2)})\, a_1^{(2)} - (a_2^{(1)} - a_2^{(2)})\, a_4 ^{(1)} -  (a_4^{(1)} - a_4^{(2)})\, a_2^{(2)} ], 
\end{multline*}
so that 
\begin{multline*}
 \frac12 \, \frac{d}{dt} \int_{\Omega} \sum_i |a_i^{(1)} - a_i^{(2)}|^2 +  \int_{\Omega}  \sum_i  d_i \, |\nabla(a_i^{(1)} - a_i^{(2)})|^2 \\
\le   \max_{i,j} \|a_i^{(j)}\|_{\infty} \, \int_{\Omega} \sum_i |a_i^{(1)} - a_i^{(2)}|^2,
\end{multline*}
and one concludes thanks to Gronwall's lemma (applied on $[0,T]$ for any $T>0$).
\hfill $\square$

\end{document}